\documentclass[10pt,letterpaper]{article}

 \usepackage{amssymb,latexsym,amsmath,amsthm}

\usepackage{fullpage}

\newtheorem*{thm*}{Theorem A}

\newtheorem{thm}{Theorem}

\newtheorem{dfn}{Definition}

\newtheorem{exam}{Example}

\newtheorem{lemma}{Lemma}

\newtheorem{remark}{Remark}

\newtheorem{cor}{Corollary}

\newtheorem{prop}{Proposition}

\author{A. Aghajani\thanks{School of Mathematics, Iran University of Science and Technology, Narmak, Tehran, Iran. Email: aghajani@iust.ac.ir.} \and
    C. Cowan\thanks{Department of Mathematics, University of Manitoba, Winnipeg, Manitoba, Canada R3T 2N2. Email: craig.cowan@umanitoba.ca. Research supported in part by NSERC.} }

\begin{document}

\def\d{ \partial_{x_j} }
\def\Na{{\mathbb{N}}}

\def\Z{{\mathbb{Z}}}

\def\IR{{\mathbb{R}}}

\newcommand{\E}[0]{ \varepsilon}

\newcommand{\la}[0]{ \lambda}

\newcommand{\s}[0]{ \mathcal{S}}

\newcommand{\AO}[1]{\| #1 \| }

\newcommand{\BO}[2]{ \left( #1 , #2 \right) }

\newcommand{\CO}[2]{ \left\langle #1 , #2 \right\rangle}

\newcommand{\R}[0]{ \IR\cup \{\infty \} }

\newcommand{\co}[1]{ #1^{\prime}}

\newcommand{\p}[0]{ p^{\prime}}

\newcommand{\m}[1]{   \mathcal{ #1 }}

\newcommand{ \W}[0]{ \mathcal{W}}

\newcommand{ \A}[1]{ \left\| #1 \right\|_H }

\newcommand{\B}[2]{ \left( #1 , #2 \right)_H }

\newcommand{\C}[2]{ \left\langle #1 , #2 \right\rangle_{  H^* , H } }

 \newcommand{\HON}[1]{ \| #1 \|_{ H^1} }

\newcommand{ \Om }{ \Omega}

\newcommand{ \pOm}{\partial \Omega}

\newcommand{\D}{ \mathcal{D} \left( \Omega \right)}

\newcommand{\DP}{ \mathcal{D}^{\prime} \left( \Omega \right)  }

\newcommand{\DPP}[2]{   \left\langle #1 , #2 \right\rangle_{  \mathcal{D}^{\prime}, \mathcal{D} }}

\newcommand{\PHH}[2]{    \left\langle #1 , #2 \right\rangle_{    \left(H^1 \right)^*  ,  H^1   }    }

\newcommand{\PHO}[2]{  \left\langle #1 , #2 \right\rangle_{  H^{-1}  , H_0^1  }}

 \newcommand{\HO}{ H^1 \left( \Omega \right)}

\newcommand{\HOO}{ H_0^1 \left( \Omega \right) }

\newcommand{\CC}{C_c^\infty\left(\Omega \right) }

\newcommand{\N}[1]{ \left\| #1\right\|_{ H_0^1  }  }

\newcommand{\IN}[2]{ \left(#1,#2\right)_{  H_0^1} }

\newcommand{\INI}[2]{ \left( #1 ,#2 \right)_ { H^1}}

\newcommand{\HH}{   H^1 \left( \Omega \right)^* }

\newcommand{\HL}{ H^{-1} \left( \Omega \right) }

\newcommand{\HS}[1]{ \| #1 \|_{H^*}}

\newcommand{\HSI}[2]{ \left( #1 , #2 \right)_{ H^*}}

\newcommand{\WO}{ W_0^{1,p}}
\newcommand{\w}[1]{ \| #1 \|_{W_0^{1,p}}}

\newcommand{\ww}{(W_0^{1,p})^*}

\newcommand{\Ov}{ \overline{\Omega}}

\title{Explicit  estimates on positive supersolutions of nonlinear elliptic equations and applications}
\maketitle

\begin{abstract}
In this paper we consider positive supersolutions of the nonlinear elliptic equation
\[- \Delta u = \rho(x) f(u)|\nabla u|^p, \qquad \hfill \mbox{ in } \Omega,\]
where $0\le p<1$, $ \Omega$ is an arbitrary   domain (bounded or unbounded) in $ \IR^N$ ($N\ge 2$), $f: [0,a_{f}) \rightarrow \Bbb{R}_{+}$ $(0 < a_{f} \leqslant +\infty)$ is a  non-decreasing continuous function and $\rho: \Omega \rightarrow \IR$ is a positive function. Using the maximum principle we give explicit estimates on positive supersolutions $u$ at each point $x\in\Omega$ where $\nabla u\not\equiv0$ in a neighborhood of $x$. As consequences, we discuss the dead core set of  supersolutions on bounded domains, and also obtain Liouville type results in unbounded domains $\Omega$ with the property that $\sup_{x\in\Omega}dist (x,\partial\Omega)=\infty$. 

\end{abstract}

\noindent
{\it \footnotesize 2010 Mathematics Subject Classification}. {\scriptsize }

\section{Introduction and main estimates}

The aim of this paper is to give explicit  estimates on   positive  classical supersolutions of the following elliptic equation
 \begin{equation}\label{original}
- \Delta u = \rho(x) f(u)|\nabla u|^p, \qquad \hfill \mbox{ in } \Omega, \\
\end{equation}
 where $ \Omega$ is an arbitrary   domain (bounded or unbounded) in $ \IR^N$, $0\le p<1$ and $f,\rho$ satisfy\\

($\mathcal{C}$) $f:D_{f}= [0,a_{f}) \rightarrow [0,\infty)$ $(0 < a_{f} \leqslant +\infty)$ is a  non-decreasing continuous function and $\rho: \Omega \rightarrow R$ is a positive function. Also we assume that $f(u)>0$  for $u>0$.\\

By a positive classical supersolution we mean a positive function $u\in C^2(\Omega)$ such  that $- \Delta u \ge \rho(x) f(u)|\nabla u|^p$, for all $x\in\Omega$.  Note in the case when $f$  in $(\ref{original})$ is not monotone we can still use our results  if we additionally have $\inf_{s> s_{0}} f(s)>0$ for every $s_{0}>0$. Indeed in this case one can take $g(t):=\inf_{s\geq t} f(s)$ then $g$ is non-decreasing and every supersolution $u$ of $(\ref{original})$ is  also a supersolution of $- \Delta u = \rho(x) g(u)|\nabla u|^p$ in $ \Omega$.\\

In this paper,  we give explicit estimates on positive classical  supersolutions $u$ of $(\ref{original})$ at each point $x\in\Omega$ where $\nabla u\not\equiv0$ in a neighborhood of $x$. As we shall see, the simplicity and robustness of our  maximum principle-based estimates provide
for their applicability to many quasi-linear elliptic inequalities on arbitrary domains in $\IR^N$, bounded or unbounded. The applications we are interested in applying the pointwise estimate to are: Liouville type theorems related to (\ref{original}) in unbounded domains such as $\IR^N$, $\IR^N_+$, exterior domains or generally unbounded domains with the property that
$\sup_{x\in\Omega}dist (x,\partial\Omega)=\infty$, and also we discuss issues related supersolutions of (\ref{original}) on bounded domains which have dead cores. \\
In particular we apply our results to the equation
 \begin{equation}\label{int2}-\Delta u= |x|^\beta u^q |\nabla u|^p,~~x\in\Omega,
\end{equation}
where $\beta\in \IR$, $q>0$, $0\le p<1$ and $\Omega$ is an arbitrary
domain in $\IR^N$; note importantly that we allow $p=0$,  and hence we obtain results regarding semilinear equations.  The existence and nonexistence of  positive
supersolutions of $(\ref{int2})$ when $\Omega$ is an exterior domain in $\IR^N$, in particular in the case $\beta=0$
and some similar  equations  have been studied
extensively  in  recent years, see
\cite{AS1,AS2,AMQ1,AMQ2,AMQ3,AMQ4,ADJT,Verons1,B,BV1,BMQ,CM,F,FQS,JS,V1}.\\ In particular, recently Burgos-PÃ‚Â´erez,  Garcia-MeliÃ‚Â´an and Quaas in \cite{BMQ} considered positive supersolutions of the equation $-\Delta u=f(u)|\nabla u|^q$ posed in exterior domains of $\IR^N$, where
$f$ is continuous in $[0, \infty)$ and positive in $(0, \infty)$ and $q > 0$. They classified supersolutions
$u$ into four types depending on the function $m(R) = \inf_{|x|=R} u(x)$
for large $R$, and give necessary and sufficient conditions in order to have supersolutions
of each of these types. As  consequences, they obtained many interesting Liouville
theorems for supersolutions depending on the values of $N, q$ and on some integrability
properties on $f$ at zero or infinity. \\
Also, very recently Bidaut-VÃ‚Â´eron,
 Garcia-Huidobro and  VÃ‚Â´eron in \cite{Verons1} obtained several important results on positive supersolutions of equation $(\ref{int2})$ (with $\beta=0$) in $\Omega\setminus\{0\}$ where $\Omega$ is an open subset of $\IR^N$ containing $0$, $p$ and $q$ are real exponents. It worth mentioning  that problem $(\ref{original})$  has
been studied in \cite{BD,F} for some more general operators when $0 < p < 1$ and $f(u)$ is essentially like $u^q$.\\

As  a simple application of our explicit estimates on supersolutions of $(\ref{original})$, we see that if $\frac{q}{1-p}>1$ then every positive supersolution of $(\ref{int2})$ is
eventually constant if
\[(N-2)q+p(N-1)< N+\beta.\]

As some other applications of our main estimates we also examine equation
(\ref{original}) for  nonlinearities like
\[f(u)=u^q+u^r~~or ~~f(u)=\max\{u^q, u^r\},~~0<q<1-p<r,\]
or singular nonlinearity $f(u)=\frac{1}{(1-u)^q}$, ($q>1$),  and discuss the dead core set of positive supersolutions in bounded domains. In particular we show that there exists a $\beta>0$ (we give the explicit value of $\beta$) such that if a domain $\Omega$ satisfies
\[\sup_{x\in\Omega}d_\Omega(x)>\beta,\]
then every positive solution $u$ of $-\Delta u\ge f(u)|\nabla u|^p$, with the above nonilnearities $f$, must be a dead core solution.

\begin{dfn}  (Dead Core Solutions). We call a non-negative nonzero solution  $u$ of (\ref{original}) a dead core solution
provided $K^0_u$
is nonempty; here $K^0_u$  denotes the interior of $K_u := \{x  \in \Omega: \nabla u(x) = 0\}.$\\
\end{dfn}

Define, for a given positive supersolution $u$ of $(\ref{original})$,
\[m_x(r)=\inf_{y\in B_r(x)}u(y)~~and~~\rho_x(r)=\inf_{y\in B_r(x)}\rho(y)~~for~~0<r<d_\Omega(x):=dist(x,\partial\Omega).\]
Note when $\Omega=R^N$ we set $d_\Omega(x)=\infty$. Also we set
\[\alpha_{N,p}=\frac{1-p}{2-p}\Big(N+\frac{p}{1-p}\Big)^{\frac{-1}{1-p}}~~\text{for}~~0\le p<1.\]

\begin{thm}    Suppose $f$ and $\rho$ satisfy $(\mathcal{C})$, $ \Omega$ is an arbitrary domain in $ \IR^N$ and $u$
 is a positive classical supersolution of $(\ref{original})$.\\
 i) If $0< p<1$ then for all $ x \in \Omega \backslash K_u^0$  we
 have
\begin{equation}\label{main-ineq1}
\int_{m_{x}(r)}^{u(x)}\frac{ds}{f(s)^{\frac{1}{1-p}}}\geq \frac{2-p}{1-p}\alpha_{N,p}\int_0^r (s\rho_{x}(s))^{\frac{1}{1-p}}  ds,~~for~~~0<r<d_\Omega(x).
\end{equation}
In particular, when $\rho\equiv1$ we get

\begin{equation}\label{main-ineq2}
\int_{m_{x}(r)}^{u(x)}\frac{ds}{f(s)^{\frac{1}{1-p}}}\geq \alpha_{N,p} r^{\frac{2-p}{1-p}},~~~0<r<d_\Omega(x).
\end{equation}
ii) When $p=0$ the above estimates are true for all $x\in\Omega$.
\end{thm}

The above theorem leads us to the following explicit estimates on supersolutions of $(\ref{original})$ depending on weather or not $f^{\frac{-1}{1-p}}\in L^{1}(0,a)$ for  $0<a<a_{f}$.

\begin{prop} Suppose $\Omega$ is an arbitrary  domain in $\IR^{N}$ and  $u$ is a positive  classical supersolution of
$(\ref{original})$   with  $0\le p<1$ in $\Omega$. If
$f^{\frac{-1}{1-p}}\in L^{1}(0,a)$ for  $0<a<a_{f}$  then
\begin{equation} \label{app1}
u(x)\geq F^{-1}\Big(\alpha_{N,p}\frac{2-p}{1-p}\int_0^r (s\rho_{x}(s))^{\frac{1}{1-p}}  ds\Big)~~for~~x\not\in K_u^0,
\end{equation}
where $F(t):=\int_{0}^{t}\frac{ds}{f(s)^{\frac{1}{1-p}}}$, $0<t<a_{f}$.\\
In particular, when $\rho\equiv1$ we have
\begin{equation}\label{app2}
u(x)\geq F^{-1}\Big(\alpha_{N,p} d_\Omega(x)^{\frac{2-p}{1-p}}\Big)~~for~~x\not\in K_u^0.
\end{equation}

As a consequence, if
\begin{equation}\label{app3}
 \alpha_{N,p}\frac{2-p}{1-p}\int_0^{d_\Omega(x),} (s\rho_{x}(s))^{\frac{1}{1-p}}  ds>||F||_{\infty}=\int_{0}^{a_{f}}\frac{ds}{f(s)^{\frac{1}{1-p}}},
\end{equation}
then in case $0<p<1$ we have $\nabla u\equiv0$ in a neighborhood
of $x$, means that $u$ is a deadcore solution, and in case $p=0$ equation $(\ref{original})$ has no any positive supersolution. When $\rho\equiv1$ then $(\ref{app3})$ reads as

\begin{equation}\label{app4}
\alpha_{N,p} d_\Omega(x)^{\frac{2-p}{1-p}}>\|F\|_{\infty}.
\end{equation}
Moreover, if $\|F\|_{\infty}<\infty$ then in  case $\Omega=\IR^N$ every positive supersolution $u$ is  constant when $0<p<1$, and there is no positive supersolution when $p=0$. Also, when $\Omega$ is an unbounded domain with the property that
\[\sup_{x\in\Omega}d_\Omega(x)=\infty\]
then every positive supersolution $u$ is  constant in the region $\Omega_M^c$ for some $M>0$ where
\[\Omega_M=\{x\in\Omega;~dist(x,\partial\Omega)<M)\},\]
while  there is no positive supersolution when $p=0$.
\end{prop}

\begin{remark} Note that using the above estimates of Preposition 1, where we assumed $f^{\frac{-1}{1-p}}\in L^{1}(0,a)$ for  $0<a<a_{f}$, we can formulate several general Liouville-type results on supersolutions of $(\ref{original})$ in unbounded domains, depending on $f$ and $\rho$, but we prefer to do this in concrete examples in  Section 3.
\end{remark}
For the case when $f^{\frac{-1}{1-p}}\not \in L^{1}(0,a)$ for $a>0$ we have

\begin{prop} Suppose $\Omega$ is an arbitrary  domain in $\IR^{N}$ and  $u$ is a positive classical  supersolution of
$(\ref{original})$  in $\Omega$.  If $f^{\frac{-1}{1-p}}\not \in L^{1}(0,a)$ for $a>0$, and $x\not\in K_u^0$, then
\begin{equation}\label{app5}
 \inf_{y\in B_{r}(x)}u(y)\leq G^{-1}\Big(\alpha_{N,p}\frac{2-p}{1-p}\int_0^r (s\rho_{x}(s))^{\frac{1}{1-p}}  ds\Big).
\end{equation}
where in this case $G$ is a positive primitive of the function $\frac{-1}{f^{\frac{1}{1-p}}}$ on $(0,a_{f})$ with $F(0)=\infty$.\\
In particular, when $\rho\equiv1$ we have
\begin{equation}\label{app6}
\inf_{y\in B_{r}(x)}u(y)\leq G^{-1}\Big(\alpha_{N,p} r^{\frac{2-p}{1-p}}\Big)~~for~~x\not\in K_u^0,~~0<r<d_\Omega(x).
\end{equation}
As a consequence, if $\Omega$ is an exterior domain in $\IR^N$ and $0<p<1$ then $u$ is eventually constant provided the following\\
\begin{equation}\label{app7}
N=2~~and~~\liminf_{r\rightarrow\infty}G^{-1}\Big(\alpha_{N,p}\frac{2-p}{1-p}\int_0^r (s\rho_{x}(s))^{\frac{1}{1-p}}  ds\Big)=0,\
\end{equation}
or
\begin{equation}\label{app8}
N>2~~and~~\liminf_{r\rightarrow\infty}r^{N-2}G^{-1}\Big(\alpha\frac{2-p}{1-p}\int_0^r (s\rho_{x}(s))^{\frac{1}{1-p}}  ds\Big)=0.
\end{equation}
Also when $\Omega$ is an exterior domain in $\IR^N$ and $p=0$ then there is no  positive supersolution provided $(\ref{app7})$ or $(\ref{app8})$ hold.
\end{prop}

\section{Proof of the main results}
\textbf{Proof of Theorem 1.}
Assume $0\le p<1$ and let   $u$  be a positive supersolution of $(\ref{original})$. Fix an $x_0\in \Omega$ with $\nabla u(x_0)\neq0$ and $0<r<d_\Omega(x_0)$. Then we have
\begin{equation}\label{thm1-1}
-\Delta u\geq \rho_{x_0}(r) f(m_{x_0}(r))|\nabla u|^p ~~~in~~B_r(x_0).
\end{equation}
Now set
\[w_r(y)=\alpha    \rho_{x_0}(r)^{\frac{1}{1-p}}    f(m_{x_0}(r))^{\frac{1}{1-p}}(r^q-|y-x_0|^q),\]
 where  $\alpha:=\alpha_{N,p} $ and $q:=\frac{2-p}{1-p}$. Then we have
 \begin{equation}\label{thm1-2}
 -\Delta w_r=\rho_{x_0}(r) f(m_{x_0}(r))|\nabla w_r|^p~~ in~B_r(x_0)~and ~w_r\equiv0~~ on ~\partial B_r(x_0).
 \end{equation}
 Here we  show that the function $u-w_r$ takes its minimum on $\partial B_r(x_0)$. When $p=0$ this is obvious by the maximum principle and that we have $-\Delta (u-w_r)\ge0$ in $B_r(x_0)$. When $0<p<1$, take an $s\in(0,1)$ and set  $v_s=u-sw_r$. We show that $v_s$ takes its minimum on $\partial B_r(x_0)$. Assume not and suppose $v_s$ takes its minimum at some $y\in B_r(x_0)$. First note  that $y\neq x_0$ because $\nabla v_s(x_0)\not=0$ by the above assumption that $\nabla u(x_0)\not=0$. Now using   $\nabla v_s(y)=0$, that implies $\nabla u(y)=s\nabla w_r(y)$, we compute, using $(\ref{thm1-1})$  and $(\ref{thm1-2})$,
$$\Delta v_s(y)\leq(s-s^p)\rho_{x_0}(r) f(m_{x_0}(r))|\nabla w_r(y)|^p,$$
and  since $s-s^p<0$ and $\nabla w_r(y)\neq0$  we  get $\Delta v_s(y)<0$, a contradiction. Hence $v_s$ takes its minimum on $\partial B_r(x_0)$. And since $v_s|_{\partial B_r(x_0)}\geq m_{x_0}(r)$, then
\[u(y)-sw_r(y)\geq m_{x_0}(r),~~~y\in B_r(x_0)\]
and since $s\in(0,1)$ was arbitrary we get
\[u(y)-m_{x_0}(r)\geq w_r(y),~~~y\in B_r(x_0).\]
Now let $0<h<r$ and $y\in B_{r-h}(x_{0})\subset B_{r}(x_{0})$. Then from the above inequality we also have
\[u(y)-m_{x_0}(r)\geq w_r(y)\geq \alpha \rho_{x_0}(r)^{\frac{1}{1-p}}   f(m_{x_0}(r))^{\frac{1}{1-p}}(r^q-(r-h)^q),~~~y\in B_{r-h}(x_0),\]
and taking infimum over $B_{r-h}(x_{0})$ we obtain
\[m_{x_0}(r-h)-m_{x_0}(r)\geq  \alpha \rho_{x_0}(r)^{\frac{1}{1-p}}   f(m_{x_0}(r))^{\frac{1}{1-p}}(r^q-(r-h)^q),\]
or
\[\frac{m_{x_0}(r-h)-m_{x_0}(r)}{h}\geq \alpha \rho_{x_0}(r)^{\frac{1}{1-p}}    f(m_{x_0}(r))^{\frac{1}{1-p}}~\frac{(r^q-(r-h)^q)}{h}.\]

Letting $h\rightarrow 0$ in the above we arrive at the following ordinary differential inequality with initial value condition
\begin{equation}\label{ode}
\left\{\begin{array}{ll} -m'_{x_{0}}(r)\geq q\alpha r^{q-1} \rho_{x_0}(r)^{\frac{1}{1-p}}   f(m_{x_0}(r))^{\frac{1}{1-p}}, & {\rm }r \in (0,d_{\Omega}(x_{0})),\\\\~~m_{x_{0}}(0)=u(x_{0})& {\rm }\  \end{array}\right.
\end{equation}
where $"'=\frac{d}{dr}"$. Dividing  inequality $(\ref{ode})$ by $f(m_{x_{0}}(r))^{\frac{1}{1-p}}$ and integrate from $0$ to $r$ we get
\[\int_{m_{x_{0}}(r)}^{u(x_{0})}\frac{ds}{f(s)^{\frac{1}{1-p}}}\geq q\alpha\int_0^r (s\rho_{x_0}(s))^{\frac{1}{1-p}}  ds,\]
that proves the estimate $(\ref{main-ineq1})$ when $\nabla u (x_0)\neq0$. To prove $(\ref{main-ineq1})$ in the case when $\nabla u(x) = 0$ but $x\not\in K_u^0$  it suffices to
take a sequence $x_n\in \Omega$ such that $\nabla u(x_n) \neq 0$ and $x_n\rightarrow x$, then write $(\ref{main-ineq1})$ for $x_n$ and let $n\rightarrow\infty$. Also note that when $p=0$ then $ K_u^0=\emptyset$ by the assumption that $f(t)>0$ for $t>0$, hence $(\ref{main-ineq1})$ is true for all $\in\Omega$.   
\hfill 
$\Box$ 

\begin{remark} The proof of Theorem 1 can be simplified when
$f$ is a $C^1$ increasing function. Indeed in this case taking
$w(y)=F(u(y))$ in $B_r(x)$  for a fixed $x\in \Omega\setminus
K^0_u$, where
\[F(t):=\int_{m_x(r)}^{t}\frac{ds}{f(s)^{\frac{1}{1-p}}},~~t>m_x(r),\]
then by the formula $\Delta F(u)=F''(u)|\nabla u|^2+f'(u)(\Delta
u)$ and the fact that $F''(t)<0$ we get
\[-\Delta w\ge \rho(y)|\nabla w|^p,~~y\in B_r(x)\]
and then proceed as above we arrive at the desired estimate.
\end{remark}

Now we give a short proof for Propositions 1 and 2. For the proof we also need the following lemma proved by J. Serrin and H. Zou in \cite{Serrin}, see also \cite {AS1}.\\
\begin{lemma}
Suppose $\{|x|>R>0\}\subset \Omega$. Let $u$ be a positive weak solution of the inequality
\begin{equation}
-\Delta u\geq 0, ~~x\in \Omega.
\end{equation}
Then there exist a constant $C=C(N,u,R)>0$ such that
\begin{equation}\label{serrin1}
u(x)\geq C|x|^{2-N},
\end{equation}
provided $N>2$, while
\begin{equation}\label{serrin2}
\liminf_{x\rightarrow\infty}u(x)>0,
\end{equation}
if $N\leq 2$.
\end{lemma}

\noindent 
\textbf{Proof of Propositions 1 and 2.}  Let $\frac{1}{f^{\frac{1}{1-p}}}\in L^{1}(0,a)$ for  $a>0$ and $F(t):=\int_{0}^{t}\frac{ds}{f(s)^{\frac{1}{\alpha}}}$. Then $F$ is an increasing function on $D_{f}$. Now if $u$ is a positive supersolution of
$(\ref{original})$ in $\Omega$ and $x\not\in K_u^0$ then from $(\ref{main-ineq1})$ we get
\begin{equation}
F(u(x))\ge F(u(x))-F(m_{x}(r))\geq \frac{2-p}{1-p}\alpha_{N,p}\int_0^r (s\rho_{x}(s))^{\frac{1}{1-p}}  ds,~~for~~~0<r<d_\Omega(x).
 \end{equation}
 Now the above inequality easily gives the desired results in Proposition 1 (note also that when $p=0$ then $K^0_u=\emptyset$ for a positive supersolution $u$ because of the assumption $f(u)>0$ for $u>0$). Similarly we get the estimates $(\ref{app5})$ and $(\ref{app6})$ in Proposition 2. To prove the rest we can simply use Lemma 1. Indeed, in an exterior domain $\Omega$ if there exists $x\not\in K^0_u$ with $|x|$ sufficiently large then we take $r=\frac{|x|}{2}$ in $(\ref{app6})$ to get
 \[\inf_{y\in B_{r}(x)}u(y)\leq G^{-1}\Big(\alpha_{N,p} r^{\frac{2-p}{1-p}}\Big).\]
By Lemma 1 there exists a constant $C$ depends only on $u,\Omega$ and $N$ such that $\inf_{y\in B_{r}(x)}u(y)\ge C r^{2-N}$ when $N>2$, and $\inf_{y\in B_{r}(x)}u(y)\ge C $ when $N=2$. Using these in the above inequality and letting $r\rightarrow\infty$ taking into account $(\ref{app7})$ and $(\ref{app8})$  we arrive a contradiction. Hence there not exists $x\not\in K^0_u$ with $|x|$ sufficiently large means that $u$ is eventually constant. 

\hfill $\Box$

\section{Applications}
In this section, as applications of the results in Section 2, we consider several concrete examples. We will frequently use the  following lemma (in particular part (ii) of the lemma). We give here a sketch of the proof, for the detailed proof one can see Lemmas 2.1 and 2.2 in \cite{BMQ}.
\begin{lemma}
Let $u$ be a positive, not eventually constant solution  of $-\Delta u\ge0$ in an exterior domain $\Omega$. Denote $I(R):=\inf_{|x|=R}u(x)$, then we have\\
(i) there exists $R_1$ large such that the function $I(R)$ is either strictly increasing, or strictly decreasing or constant in $(R_1, +\infty)$.\\
(ii) $I(R)$ is bounded when $N\ge3$. Also when $N=2$ we have
\[I(R)\le C \log R.\]
\begin{proof}
We give just a proof for part(ii), to prove (i) see lemma 2,1 in \cite{BMQ}.\\
Define for $x\in A(R_1,R_2):=\{x\in R^N;~R_1<|x|<R_2\}$
\[\Phi_1(x)=\frac{I(R_1)-I(R_2)}{R_1^{2-N}-R_2^{2-N}}(|x|^{2-N}-R_2^{2-N})+I(R_2),\]
when $N\ge3$.  We then see that $\Phi_1$ is a harmonic function vanishing on the $\partial A(R_1,R_2) $, hence by the maximum principle we have $u\ge \Phi_1$ in $ A(R_1,R_2)$. Now assume  $I(R)$ is not bounded and fix an $R\in(R_1,R_2)$ then
\[I(R)\ge \frac{I(R_1)-I(R_2)}{R_1^{2-N}-R_2^{2-N}}(R^{2-N}-R_2^{2-N})+I(R_2)\rightarrow\infty~~as~~R_2\rightarrow\infty,\]
a contradiction.\\
In the case $N=2$, taking

\[\Phi_2(x):=\frac{I(R_1)-I(R_2)}{\log R_1-\log R_2}(\log|x|-\log R_2)+I(R_2),\]
we see that  $\Phi_2$ is harmonic functions vanishing on the $\partial A(R_1,R_2) $ and similar as above, $u\ge\Phi_2$ in $A(R_1,R_2) $ hence for $R\in(R_1,R_2)$ we get
\begin{eqnarray*}
I(R) &\ge& \frac{I(R_1)-I(R_2)}{\log R_1-\log R_2}(\log R-\log R_2)+I(R_2)\\
&=& \left(\frac{\log R_1-\log R_2}{\log R-\log R_2}\right)   I(R_1)+\left(\frac{\log R_1-\log R}{\log R_1-\log R_2}\right)   I(R_2) \\
&\ge& \left(\frac{\log R_1-\log R}{\log R_1-\log R_2}\right)   I(R_2).
\end{eqnarray*}
Now fix $R_1$ and $R=2R_1$, then for  $R_2$ large we see from the above that
\[I(R_2)\le C\log R_2,~~\text{for $R_2$ sufficiently large}.\]
\end{proof}

\end{lemma}
\subsection{Liouville-type results}

Consider the equation
\begin{equation}\label{ex1}
-\Delta u= u^q |\nabla u|^p,~~x\in\Omega,
\end{equation}
where $\Omega$ is an arbitrary domain in $\IR^N$.\\

\begin{thm}  Let $u$ be a positive supersolution of $(\ref{ex1})$, where $q\ge0$ and $0 \le p<1$. Then\\
i) if $\frac{q}{1-p}<1$ then
\begin{equation}\label{ex1-1}
u(x)\geq \Big(\frac {\alpha(1-p-q)}{1-p}\Big)^{\frac {1-p}{1-p-q}}
r^{\frac{2-p}{1-p-q}},~~0<r<d_\Omega(x)~~x\not\in K_u^0.
\end{equation}
In particular, when $\Omega=\IR^N$ then $u$ is constant. Also, when $\Omega=B_R^c$ is an exterior domain then $u$ is constant in $B_{R'}^c$ for some $R'>R.$
Moreover, if $\Omega$ is  unbounded  with the property that
\begin{equation}\label{unbounded}
\sup_{x\in\Omega}d_\Omega(x)=\infty,
\end{equation}
and $u$ is  bounded, then  there exists an $M>0$ so that $u$ is
constant in the region $\Omega_M^c$ where
\[\Omega_M=\{x\in\Omega;~dist(x,\partial\Omega)<M)\}.\]  See Example \ref{opt_exam} to see (at least in the case of $p=0$) that one does indeed need to assume $u$ is bounded.  \\
ii) If $\frac{q}{1-p}>1$ then for all $x\not\in K_u^0$ we have
\begin{equation}\label{ex1-2}
\inf_{y\in B_{r}(x)}u(y)\leq C
r^{\frac{p-2}{p+q-1}},\qquad 0<r<d_\Omega(x).
\end{equation}
In particular if $\Omega$ is an exterior domain then $u$ is
eventually constant if 
\begin{equation}\label{ex1-3}
(N-2)q+(N-1)p< N
\end{equation}
iii) If $\frac{q}{1-p}=1$ then  for all $x\not\in K_u^0$ we have
\begin{equation}\label{ex1-0}
u(x)\geq m_x(r) e^{\alpha_{N,p}
r^{\frac{2-p}{1-p}}}, \quad ~~0<r<d_\Omega(x).
\end{equation}
In particular, when $\Omega=\IR^N$ then $u$ is constant. Also, when $\Omega=B_R^c$ is an exterior domain then $u$ is constant in $B_{R'}^c$ for some $R'>R.$\\
\end{thm}

\begin{proof}  i) Note equation $(\ref{ex1})$ is of the form of equation
$(\ref{original})$ with $f(u)=u^q$ and $\rho(x)\equiv 1$.
Hence if $\frac{q}{1-p}\neq1$ then by Theorem 1 we get
\begin{equation}\label{estim1}
\frac{1-p}{1-p-q}\Big (u(x)^{\frac{1-p-q}{1-p}}-m_x(r)^{\frac{1-p-q}{1-p}}\ge\alpha_{N,p} r^{\frac{2-p}{1-p}},~~0<r<d_\Omega(x),
\end{equation}
for all $x\not\in K_u^0$.\\ Now if
$\frac{q}{1-p}<1$ then we get the following pointwise estimate
\begin{equation}\label{estim2}
u(x)\geq \Big(\frac {\alpha(1-p-q)}{1-p}\Big)^{\frac {1-p}{1-p-q}} r^{\frac{2-p}{1-p-q}},~~0<r<d_\Omega(x)~~x\not\in K_u^0.
\end{equation}
Now if $\Omega=\IR^N$ then the assertion of theorem is obvious. Because in this case we have $d_\Omega(x)=\infty$ for every $x\in\Omega$ and then from the estimate  $(\ref{estim2})$ we must have $ K_u^0=\Omega$ means that $u$ is constant. To
prove the second part of (i), for simplicity take $\Omega=B_1^c$ and let $I(R):=\inf_{|x|=R}u(x)$. Assume that $u$ is not eventually constant then for any $R>1$ there exists $x\not\in K^0_u$ with $|x|\ge R$, and then from $(\ref{estim2})$ we have $u(x)\ge CR^{\frac{2-p}{1-p-q}}$, where $C$ is a constant independent of $R$ for all large $R$. Now by the continuity of $u$ and the fact that $u$ is constant in every connected component of $K^0_u$ we see that we have $u(x)\ge CR^{\frac{2-p}{1-p-q}}$ for all $x$ with $|x|\ge R$, thus $I(R)\ge CR^{\frac{2-p}{1-p-q}}\rightarrow\infty$ as $R\rightarrow\infty$ which contradicts Lemma 2.\\ 

Now let $\Omega$ satisfy $(\ref{unbounded})$. For an $x\in\Omega$ with $\nabla
u(x)\neq0$ and $R < d_\Omega(x)$ we get, from
$(\ref{ex1-1})$,
$u(x) \ge Cd_\Omega(x)^{\frac{2-p}{1-p-q}}$. Now assume
$u$ is a bounded solution but the assertion is not true, then
there exists a sequence $\{x_n\}\subset \Omega$ such that $d_\Omega(x_n) \rightarrow\infty$ with $\nabla u(x_n)\neq0$, then
from the above we get $u(x_n) \ge Cd_\Omega(x_n)^{\frac{2-p}{1-p-q}}$
implies that u is unbounded, a contradiction.\\

\noindent
ii) Now assume $\frac{q}{1-p}>1$. Then by Proposition 2, or from $(\ref{estim1})$,  we get
\begin{equation} \label{ext}
\inf_{y\in B_{r}(x)}u(y)\leq C r^{\frac{p-2}{p+q-1}},~~x\not\in K_u^0.
\end{equation}
Now let $\Omega$ be an exterior domain in $R^N$ and assume that
$u$ is not eventually constant then for any $R > 0$ large there
exists $x\not\in K_u^0$ with $|x|>R$. If $N=2$ then $(\ref{ext})$
contradicts $(\ref{serrin2})$, hence $u$ is eventually constant.
Also when $N\geq 3$   we have $u(x)\ge C|x|^{2-N}$ by
$(\ref{serrin1})$, hence we need $N-2\ge \frac{2-p}{p+q-1},$ or
equivalently $(N-2)q+p(N-1)\ge N.$ Thus $u$ is eventually
constant if we have
\[(N-2)q+p(N-1)< N.\]

\noindent
(iii) Now assume $\frac{q}{1-p}=1$, then by Theorem 1 and that $\int_{m_x(r)}^u \frac{ds}{f(s)^{\frac{1}{1-p}}}=\ln \frac{u}{m_x(r)}$ we get
\[u(x)\geq m_x(r) e^{\alpha_{N,p} r^{\frac{2-p}{1-p}}},~~0<r<d_\Omega(x),
~~x\not\in K_u^0.\]
Now if $\Omega=\IR^N$ then by Lemma 1 we get $m_x(r)>C>0$ when $N=2$, and $m_x(r)>C|r|^{2-N}$  when $N\ge3$. Now using the fact that for arbitrary $\beta>0$ we have $e^{\alpha_{N,p} r^{\frac{2-p}{1-p}}}\ge r^\beta $ for large $r$, then  letting $r\rightarrow\infty$ in the above estimates easily gives  that $u$ is constant. A similar argument as in part (i) for the case when $\Omega $ is an exterior domain shows that  $u$ is eventually constant. Note in this case we also use the fact that by Lemma 1 we have, for $r=\frac{|x|}{2}$ with $|x|$ sufficiently large,  $m_x(r)>C>0$ when $N=2$, and $m_x(r)>C|r|^{2-N}$  when $N\ge3$. \\
\end{proof}

\begin{exam} \label{opt_exam}  Let $ 0<q<1$ and $ S \subset S^{N-1}$ a smooth subset with nonempty boundary.   For $ r=|x|$ and $ \theta = \frac{x}{|x|}$ we consider the cone $ \Omega:=\{ x \in \IR^N:   r>0, \theta \in S \}$.  Then   $ u(x)=u(r,\theta)= r^\frac{2}{1-q} w(\theta)$  is a bounded positive classical solution of $ -\Delta u = u^q$ in $ \Omega$ with $ u=0$ on $ \pOm$ exactly when $ w>0$ is a bounded classical solution of
\begin{equation} \label{w_eq}
-\Delta_\theta w - \beta_q w = w^q \quad \mbox{ in } S, \qquad \quad w=0 \mbox{ on } \partial S,
\end{equation}  where
$\beta_q=\frac{2}{1-q} \left(N-2 + 2 \frac{1}{1-q} \right)$ and $\Delta_\theta$ is the Laplace-Beltrami operator on $S^{N-1}$.  
Consider the energy
\[ E(w):= \int_S  \frac{| \nabla_\theta w|^2}{2} - \frac{\beta_q}{2} w^2 - \frac{|w|^{q+1}}{q+1} d \theta. \]  Provided $ \beta_q< \lambda_1(S)$ (the first eigenvalue of $ -\Delta_\theta$ in $H_0^1(S)$) and since $q<1$ one sees there is a nonzero minimizer $ w$ of $ E$ over $H_0^1(S)$ and one can take $ w>0$.  After standard arguments one sees this $w$ is a positive bounded classical solution of (\ref{w_eq}).
\end{exam}

Now consider the more general inequality
\begin{equation}\label{liouville2}
-\Delta u= |x|^\beta u^q |\nabla u|^p,~~x\in\Omega,\end{equation}
where $\beta \in \IR$ and $\Omega$ is an exterior domain  in $\IR^N$. For simplicity let $\Omega=\IR^N\setminus B_1$.

\begin{thm}
Let $u$ be a positive supersolution of $(\ref{liouville2})$ in
$\Omega=\IR^N\setminus B_1$, $q\ge0$, $0\le p<1$ and $\beta\in \IR$.\\
i) If $\frac{q}{1-p}\le1$ then every positive supersolution is
eventually constant if $\beta\ge p-2$.\\
ii) If $\frac{q}{1-p}>1$ then every positive supersolution is
eventually constant if
\[(N-2)q+p(N-1)< N+\beta.\]
\end{thm}

\begin{proof}
(i) First note that when $\frac{q}{1-p}<1$ then the case  $\beta>0$ is not interesting as in this case we have  $|x|^\beta\ge 1$ and then $u$ is also a supersolution of  $(\ref{ex1})$. So assume $\beta<0$, then taking $\rho(x):=|x|^\beta$ we have
\[\rho_x(r)=\inf_{B_r(x)} \rho(y)=(|x|+r)^{\beta},~~0<r<d_\Omega(x)=|x|-1.\]
Then for $x\not\in K_u^0$ we have
\[ \int_0^r (s\rho_{x}(s))^{\frac{1}{1-p}}  ds=\int_0^r [s(|x|+s)^\beta]^{\frac{1}{1-p}}  ds=|x|^{\frac{2+\beta-p}{1-p}} \int_0^{\frac{r}{|x|}}[t(1+t)^\beta]^{\frac{1}{1-p}}. \]
Now from Proposition 1 we get the following explicit estimate at every $x\not\in K_u^0$

\[u(x)\geq \Big(\frac {\alpha(1-p-q)}{1-p}\Big)^{\frac {1-p}{1-p-q}} |x|^{\frac{2+\beta-p}{1-p-q}} \Big(\int_0^{\frac{|x|-1}{|x|}}[t(1-t)^\beta]^{\frac{1}{1-p}} \Big)^{\frac{1-p}{1-p-q}},~~x\not\in K_u^0.\]
In particular, for any $\gamma>1$ we get, for $x\not\in K_u^0$
\[u(x)\geq C_{p,N,\gamma}  |x|^{\frac{2+\beta-p}{1-p-q}},~~~x\in \IR^N\setminus B_\gamma,\]
where
\[C_{p,N,\gamma}=\Big(\frac {\alpha(1-p-q)}{1-p}\Big)^{\frac {1-p}{1-p-q}}\Big(\int_0^{\frac{\gamma-1}{\gamma}}[t(1-t)^\beta]^{\frac{1}{1-p}} \Big)^{\frac{1-p}{1-p-q}}.\]
Thus similar as in part (i) of Theorem 2, we see that $u$ is eventually constant if $\frac{2+\beta-p}{1-p}>0$ or equivalently $\beta>p-2$. Also,
similar to the proof of part (iii)   of Theorem 2, one can treat the case $\frac{q}{1-p}=1$ using the above estimate on $\rho(x)$ to show $u$ is eventually constant if  $\beta>p-2$.\\

\noindent 
(ii) Assume $\frac{q}{1-p}>1$. If $\beta<0$ then using the above computations on $\rho_x(r)$ and Proposition 2 we get, for  $x\not\in K_u^0$ and $\gamma>1$
\[\inf_{y\in B_{r}(x)}u(y)\leq C_{p,N,\gamma} r^{\frac{p-2-\beta}{p+q-1}},~~~x\in \IR^N\setminus B_\gamma.\]
Then similar as the proof of part (ii) of Theorem 2, where we also used Lemma 1, we see that
when $N= 2$ and $p-2<\beta$ then $u$ must be eventually constant.
 Also when $N\geq 3$, if there exists $x\not\in K_u^0$ with $|x|$ sufficiently large,   then using the estimate $u(x)\ge C|x|^{2-N}$ for superharmonic functions in exterior domains by Lemma 1, we must have
\[N-2\ge \frac{2+\beta-p}{p+q-1},\]
or equivalently
\[(N-2)q+p(N-1)\leq N+\beta.\]
Thus $u$ is eventually constant if we have
\[(N-2)q+p(N-1)< N+\beta.\]
Now consider the case ($\beta>0$). Here we have
$\rho_{x}(r)=(|x|-r)^{\beta}>(\frac{3}{2})^{\beta}|x|^{\beta}$ for
$0<r<\frac{|x|}{2}$ and
similar  as above we will see that if $(N-2)q+p(N-1)< N+\beta$ then $u$ is eventually constant .\\

\end{proof}

\subsection{Deadcore supersolutions in bounded domains}
Consider  equation (1) with $\rho\equiv1$ and \[f(u)=u^q+u^r~~or ~~f(u)=\max\{u^q, u^r\}\]
 with $0<q<1-p<r$. Then we see that in both cases we have $f^{\frac{1}{1-p}}\in L^1(0,a)$ for every $a>0$ and also $F(\infty)=\int_0^\infty \frac{ds}{f(s)^{\frac{1}{1-p}}}<\infty$.

For example consider (1) with $f(u)=\max\{u^q, u^r\}$, i.e.,
\begin{equation}\label{deadeq}
-\Delta u=\max\{u^q, u^r\} |\nabla u|^p,~~x\in\Omega.
\end{equation}
As a corollary of Proposition 1 we have
\begin{cor} Let $u$ be a positive supersolution of $(\ref{deadeq})$ in an arbitrary domain $\Omega$ (bounded or not) and $0<q<1-p<r$. Then $u$ is a dead core solution if
\[\sup_{x\in\Omega}d_\Omega(x)> \Big(\frac{1-p}{1-p-q}+\frac{1-p}{p+r-1}\Big)^{\frac{1-p-q}{1-p}}=:\beta.\]
In particular, if $\Omega=B_R$ with $R>\beta$ then $B_{R-\beta}$ is a dead core set of any solution $u$. Also when  $\Omega=\IR^N\setminus B_1$ then  $\Omega=\IR^N\setminus B_{1+\beta}$ is a dead core set of any solution $u$.
\end{cor}
\begin{proof}
By the notation of Proposition 1   we have

\[F(\infty)=\int_0^\infty \frac{ds}{f(s)^{\frac{1}{1-p}}}=\int_0^1 \frac{ds}{s^{\frac{q}{1-p}}}+\int_1^\infty \frac{ds}{s^{\frac{r}{1-p}}}=\frac{1-p}{1-p-q}+\frac{1-p}{p+r-1}.\]
Then by Proposition 1 we see that  any positive supersolution $u$ must be constant on the set
\[\mathcal{S}:=\{x\in\Omega;~d_\Omega(x)^{\frac{2-p}{1-p}}\geq F(\infty)\}=\{x\in\Omega;~d_\Omega(x)\geqslant \beta\},\]
where
\[\beta:=\Big(\frac{1-p}{1-p-q}+\frac{1-p}{p+r-1}\Big)^{\frac{1-p-q}{1-p}}.\]
This in particular shows that if a domain $\Omega$ satisfies
\[\sup_{x\in\Omega}d_\Omega(x)>\beta,\]
then every supersolution $u$ of ($(\ref{deadeq})$ is a dead core supersolution. Now when $\Omega=B_R$ with $R>\beta$ then we have
\[B_{R-\beta}\subset \{x\in\Omega;~d_\Omega(x)\geqslant \beta\},\]
hence for any supersolution $u$ we have
\[u\equiv C~~on~~B_{R-\beta}.\]
Similarly when  $\Omega=\IR^N\setminus B_1$, then
for every supersolution $u$ we must have
\[u\equiv C~~~on~~|x|\geq 1+\beta.\]

\end{proof}

Now consider the equation
\begin{equation}\label{singular}
-\Delta u= \frac{|\nabla u|^p}{(1-u)^q},~~x\in\Omega~~(q>1>p>0)
\end{equation}
which is of the form $(\ref{original})$
with \textbf{singular} nonlinearity
\[f(u)=\frac{1}{(1-u)^q},~~q>1.\]
We have

\begin{cor} Let $0\le u<1$ be a positive supersolution of $(\ref{singular})$ in an arbitrary domain $\Omega$ (bounded or not), where $0<p<1<q$. Then $u$ is a dead core supersolution if
\[\sup_{x\in\Omega}d_\Omega(x)> \Big(\frac{1-p}{\alpha (1+q-p)}\Big)^{\frac{1-p}{2-p}}\]
\end{cor}
\begin{proof}
Here we have
\[F(t)=\int_0^t \frac{ds}{f(s)^{\frac{1}{1-p}}}=\int_0^t (1-s)^{\frac{q}{1-p}}=\frac{1-p}{1+q-p}[1-(1-t)^{\frac{1+q-p}{1-p}}].\]
Hence,  if $0<u<1$ is a supersolution of $(\ref{singular})$ then by the above results we get
\[(1-u(x))^{\frac{1+q-p}{1-p}}\leq 1- \frac{\alpha (1+q-p)}{1-p}d_\Omega(x)^{\frac{2-p}{1-p}},~~x\not\in K_u^0.\]
This in particular shows that if
\[d_\Omega(x)\geq \beta:= \Big(\frac{1-p}{\alpha (1+q-p)}\Big)^{\frac{1-p}{2-p}},\]
then $x\in K_u^0$. Moreover, if
\[\sup_{x\in\Omega}d_\Omega(x)> \beta,\]
then every supersolution $0<u<1$ of $(\ref{singular})$ is a dead core supersolution.\\

\end{proof}


\end{document}